\documentclass[review]{elsarticle}

\usepackage{lineno,hyperref}
\usepackage{amssymb,amsmath,amsthm, amsfonts, mathrsfs, bbm, dsfont}
\usepackage[utf8]{inputenc}
\newtheorem{theorem}{Theorem}[section]
\newtheorem{lemma}{Lemma}[section]

\newtheorem{corollary}{Corollary}[section]

\newtheorem{definition}{Definition}[section]

\newtheorem{remark}{Remark}

\usepackage{extarrows}
\modulolinenumbers[5]

\begin{document}

\begin{frontmatter}

\title{ Closed subspaces and some basic topological properties of noncommutative Orlicz spaces\tnoteref{mytitlenote}}
\tnotetext[mytitlenote]{The research is supported by National Science Foundation of China (Grant No. 11371222 ) and by the scientific research foundation of Education Bureau
of Hebei Province and by Graduate student innovation program of Beijing Institute of Technology (Grant No. 2015CX10037 ).}


\author[mymainaddress]{ Lining Jiang}
\ead{jianglining@bit.edu.cn}

\author[mymainaddress,mysecondaryaddress]{Zhenhua Ma\corref{Zhenhua Ma } }
\cortext[Zhenhua Ma ]{Corresponding author}
\ead{mazhenghua\_1981@163.com}

\address[mymainaddress]{School of Mathematics and Statistics, Beijing Institute of Technology, Beijing, 100081, P. R. China}
\address[mysecondaryaddress]{Department of Mathematics and Physics, Hebei University of Architecture, Zhangjiakou, 075024, P. R. China}

\begin{abstract}
In this paper, we study the noncommutative Orlicz space $L_{\varphi}(\widetilde{\mathcal{M}},\tau)$, which generalizes the concept of noncommutative $L^{p}$ space, where $\mathcal{M}$ is a von Neumann algebra, and $\varphi$ is an Orlicz function. As a modular space, the space $L_{\varphi}(\widetilde{\mathcal{M}},\tau)$ possesses the Fatou property, and consequently, it is a Banach space. In addition, a new description of the subspace $E_{\varphi}(\widetilde{\mathcal{M}},\tau)=\overline{\mathcal{M}\bigcap L_{\varphi}(\widetilde{\mathcal{M}},\tau)}$ in $L_{\varphi}(\widetilde{\mathcal{M}},\tau)$, which is closed under the norm topology and dense under the measure topology, is given. Moreover, if the Orlicz function $\varphi$ satisfies the $\Delta_{2}$-condition, then $L_{\varphi}(\widetilde{\mathcal{M}},\tau)$ is uniformly monotone, and the convergence in the norm topology and measure topology coincide on the unit sphere. Hence, $E_{\varphi}(\widetilde{\mathcal{M}},\tau)=L_{\varphi}(\widetilde{\mathcal{M}},\tau)$ if $\varphi$ satisfies the $\Delta_{2}$-condition.
\end{abstract}


\end{frontmatter}

\linenumbers

\section{Preliminaries}
Noncommutative integration theory was first introduced by I.E. Segal \cite{Segal} and is a fundamental tool in many theories, such as operator theory and noncommutative probability theory. Since the noncommutative space $L^{p}(\mathcal{M},\tau)$ \,\,$(1\leq p\leq \infty)$ , where $\tau$ is a faithful semifinite normal trace on a von Neumann algebra $\mathcal{M}$, has been defined \cite{Yeadon,Segal}, many scholars have conducted a systematic study of these spaces, and obtained many interseting results \cite{Xu1,Xu2,Haagerup}. As a natural extension of the space, the theory of noncommutative Orlicz space associated to a trace was introduced and studied by many mathematicians. For details, one can see \cite{Muratov}, \cite{Rashed}, \cite{Ghadir}, \cite{Kunze}, and so on. In this paper we will take Sadeghi's approach \cite{Ghadir} and study the topological properties of noncommutative Orlicz spaces.

To begin with, we collect some definitions and facts related to von Neumann algebras.
Suppose that $\mathcal{M}$ is a semi-finite von Neumann algebra acting on a Hilbert space $\mathcal{H}$ with a normal semi-finite faithful trace $\tau.$ The identity in $\mathcal{M}$ is denoted by $\mathbf{1}$ and the set of all self-adjoint projections on $\mathcal{M}$ is denoted by $\mathcal{P}(\mathcal{M})$.

\begin{definition}A densely-defined closed linear operator $x: \mathcal{D}(x)\rightarrow \mathcal{H}$ with domain $\mathcal{D}(x)\subseteq\mathcal{H}$ is called affiliated with $\mathcal{M}$ if and only if $u^{\ast}xu=x$ for all unitary operators $u$ belonging to the commutant $\mathcal{M^{\prime}}$ of $\mathcal{M}$.
\end{definition}

\begin{definition}\cite{Nelson}
Suppose that $x$ affiliated with $\mathcal{M}$. We call $x$ is $\tau$-measurable, if there exists a number $\lambda\geq0$ such that $$\tau(e_{(\lambda,\infty)}(|x|))<\infty,$$
 where $e_{(\lambda,\infty)}(|x|)$ is the spectral projection of $|x|$ corresponding to the interval $(\lambda,\infty)$. The collection of all $\tau$-measurable operators is denoted by $\widetilde{\mathcal{M}}$.
\end{definition}

Given $0<\varepsilon,\delta\in \mathds{R}$, set
\begin{eqnarray*}
 \mathcal{V}(\varepsilon,\delta)=\{x\in\widetilde{\mathcal{M}}: \,\, there \,\, exists \,\, e\in\mathcal{P}(\mathcal{M})
 \,\,such\,\,that\,\, \\e(\mathcal{H})\in\mathcal{D}(x), \|xe\|_{\mathcal{B(\mathcal{H})}}\leq\varepsilon \,\,and \,\, \tau(\mathbf{1}-e)\leq\delta\}.
\end{eqnarray*}
Here, $\varepsilon,\delta$ run over all strictly positive numbers \cite{Nelson}. An alternative description of the set is given by
$$\mathcal{V}(\varepsilon,\delta)=\{x\in\widetilde{\mathcal{M}}: \tau(e_{(\varepsilon,\infty)}(|x|))<\delta\}.$$

\begin{definition}(\cite{Nelson}) Suppose that $x_n,\,x\in \widetilde{\mathcal{M}}$. We say that $x_{n}$ converges to $x$ in measure ($x_{n}\xrightarrow{\tau_{m}}x$ for short ), if for all $\varepsilon,\delta>0$, there exists an $n_{0}$ such that $x_{n}-x\in\mathcal{V}(\varepsilon,\delta)$, $n\geq n_{0}$.
\end{definition}
 \begin{remark}
 1) Using the definition of $\mathcal{V}(\varepsilon,\delta)$, one can get that $x_{n}\xrightarrow{\tau_{m}}x$ if and only if $$\lim_{n\rightarrow\infty}\tau(e_{(\varepsilon,\infty)}(|x_{n}-x|))=0$$
 for any $\varepsilon>0.$

 2) It is known that the collection $\{\mathcal{V}(\varepsilon,\delta)\}_{\varepsilon,\delta>0}$ is a neighborhood base at 0 for a vector space topology $\tau_{m}$ on $\widetilde{\mathcal{M}}$ and that $\widetilde{\mathcal{M}}$ is a complete topological $\ast$-algebra.
 \end{remark}

In the setting of $\tau$-measurable operators, the generalized singular value functions are the analogue (and actually, generalization) of the decreasing rearrangements of functions in the classical settings, and is more importantly the cornerstone for the theory of noncommutative rearrangement invariant Banach function spaces \cite{Dodds1}.

\begin{definition}(\cite{Fack})
For $x\in\widetilde{\mathcal{M}}$, the distribution function $\lambda_{(\cdot)}(x): [0,\infty)\rightarrow[0,\infty]$ is defined by
$$\lambda_{s}(x)=\tau(e_{(s,\infty)}(|x|)),\,\,s\geq0.$$
\end{definition}

Since the operator $x$ is $\tau$-measurable, $\lambda_{s}(x)<\infty$ for $s$ large enough and $\lim_{s\rightarrow\infty}\lambda_{s}(x)=0$ as noted before. Furthermore, the function $\lambda_{s}(x)$ is decreasing and right-continuous since $\tau$ is normal and $e_{(s_{n},\infty)}(|x|)\uparrow e_{(s,\infty)}(|x|)$ strongly as $s_{n}\downarrow s.$

\begin{definition}(\cite{Fack})
Let $L^{0}(X,\Sigma, m)$ be the space of measurable functions on some $\sigma$-finite measure space $(X,\Sigma, m)$. Give an element $x\in\widetilde{\mathcal{M}}$, the generalized singular value function $\mu_{(\cdot)}(x): [0,\infty]\rightarrow[0,\infty]$ is defined by
$$\mu_{t}(x)=\inf\{s\geq0: \lambda_{s}(x)\leq t \}, \,\, t>0,$$
where $\lambda_{s}(x)$ is the distribution function.
\end{definition}
It is known that the infimum can be attained and that $ \lambda_{\mu_{t}(x)}(x)\leq t,\,\,t>0$. For details on the generalised singular value see \cite{Fack}. We proceed to briefly review the concept of a Banach function space of measurable functions on $(0,\infty)$ (see \cite{Dodds}.) A function norm $\rho$ on $L^{0}(0,\infty)$ is defined to be a mapping $\rho: L^{0}_{+}\rightarrow [0,\infty]$ satisfying

1. $\rho(f)=0$ iff $f=0$ a.e.

2. $\rho(\lambda f)=\lambda\rho(f)$ for all $f\in L^{0}_{+}, \lambda>0$.

3. $\rho(x+y)\leq \rho(x)+\rho(y)$ for all.

4. $f\leq g$ implies $\rho(f)\leq\rho(g)$ for all $f,g \in L^{0}_{+}$.

Such a $\rho$ may be extended to all of $L^{0}$ by setting $\rho(f)=\rho(|f|)$, in which case we may define $L^{\rho}(0,\infty)=\{f\in L^{0}(0,\infty): \rho(f)<\infty\}$. It now $L^{\rho}(0,\infty)$ turns out to be a Banach space when equipped with the norm $\rho(\cdot)$, we refer to it as a Banach function space.

Using the above context Dodds, Dodds and de Pagter \cite{Dodds} formally defined the noncommutative space $L^{\rho}(\widetilde{M})$ to be
$$L^{\rho}(\widetilde{M})=\{f\in \widetilde{M}: \mu(f)\in L^{\rho}(0,\infty)\}$$
and showed that if $\rho$ is lower semicontinuous and $L^{\rho}(0,\infty)$ rearrangement-invariant, $L^{\rho}(\widetilde{M})$ is a Banach space when equipped with the norm $\|f\|_{\rho}=\rho(\mu(f))$.

\begin{remark}
If $\mathcal{M}$ is a commutative von Neumann algebra, then $\mathcal{M}$ can be identified with $L^{\infty}(X,\mu)$ and $\tau(f)=\int_{X} f \mathrm{d}\mu$, where $(X,\mu)$ is a localizable measure space, and where the distribution function and the generalized singular value function defined above are exactly the usual distribution function and classical rearrangement \cite{Stein}.
\end{remark}

Next we recall the definition and  some basic properties of noncommutative Orlicz spaces.

\begin{definition}(\cite{Chen})
The function $\varphi: [0,\infty)\rightarrow[0,\infty]$ is called an Orlicz function if
$$\varphi(u)=\int^{|u|}_{0}p(t) \mathrm{d}t, $$
where the real-valued function $p$ defined on $[0,\infty)$ has the following properties:

$(1)$ $p(0)=0, p(t)>0$ for $t>0$ and $\lim_{t\rightarrow\infty}p(t)=\infty;$

$(2)$ $p$ is right continuous;

$(3)$ $p$ is nondecreasing on $(0,\infty).$
\end{definition}

For every Orlicz function $\varphi$, there is a complementary Orlicz function $\psi: [0,\infty)\rightarrow[0,\infty]$ defined by
$$\psi(u)=\sup\{uv-\varphi(v): v\geq0\}.$$

A pair of complementary Orlicz functions $(\varphi,\psi)$ fulfils the following Young Inequality:
$$uv\leq\varphi(u)+\psi(v),\,\, u,v\in[0,\infty),$$
and equality holds if and only if $u=\psi(v)$ or $v=\varphi(u)$.
 For background on Orlicz functions and Orlicz spaces one can see \cite{Rao,Chen}.

Suppose that $\varphi$ is an Orlicz function. For $x\in\widetilde{\mathcal{M}}$, set
$$\widetilde{\rho}_{\varphi}(x)=\tau(\varphi(|x|)),$$
then $\tau(\varphi(|x|))$ is a convex modular on $\widetilde{\mathcal{M}}$ \cite{Ghadir}.

\begin{definition}Set
$$L_{\varphi}(\widetilde{\mathcal{M}},\tau)=\left\{x\in\widetilde{\mathcal{M}}:  \tau\left(\varphi(\lambda|x|)\right)<\infty \,\, for\,\,some\,\,\lambda>0\right\},$$
and equip the space with the Luxemburg norm
$$\|x\|=\inf\left\{\lambda>0: \tau\left(\varphi\left(\frac{|x|}{\lambda}\right)\right)\leq1\right\}.$$
Such a space is called a noncommutative Orlicz space.
\end{definition}

\begin{remark}
Notice that if $\varphi(x)=|x|^{p},\,\,1\leq p<\infty$ for any $\tau$-measurable operator $x\in{\widetilde{\mathcal{M}}}$, then $L_{\varphi}(\widetilde{\mathcal{M}},\tau)$ is nothing but the noncommutative $L^p$ space $L^{p}(\widetilde{\mathcal{M}},\tau)$ and the Luxemburg norm generated by this function is expressed by the formula
$$\|x\|_{p}=\left(\tau\left(|x|^{p}\right)\right)^{\frac{1}{p}}.$$
\end{remark}

Similar to the commutative case, for $x,y\in L_{\varphi}(\widetilde{\mathcal{M}},\tau)$ one can define the following Orlicz norm:
$$\|x\|^{o}=\sup\{\tau(|xy|): \tau(\psi(|y|))\leq1\},$$
where $\psi$ is the complementary function of $\varphi.$
Moreover, we have the following relation between the two norms \cite{Ghadir},
$$\|x\|\leq\|x\|^{o}\leq2\|x\|.$$

We also can get the Young Inequality in noncommutative Orlicz spaces:
\begin{lemma}(\cite{Kunze})
For a pair $(\varphi,\psi)$ of complementary Orlicz functions we have:
$$\tau(|xy|)\leq\tau(\varphi(|x|))+\tau(\psi(|y|)),\,\, for\,\, all\,\, x,y\in \widetilde{\mathcal{M}}.$$
Moreover, if $0\leq x\in\widetilde{\mathcal{M}}$ with $\tau(\varphi(x))<\infty$, then there is a $0\leq y\in\widetilde{\mathcal{M}}$ with
$$\tau(xy)=\tau(\varphi(x))+\tau(\psi(y))\,\, and\,\,\tau(\psi(y))\leq1.$$
\end{lemma}
For further information on the theory of noncommutative Orlicz spaces we refer the reader to \cite{Muratov,Rashed,Rashed1,Ghadir,Kunze}.

\section{Closed linear subspaces of $L_{\varphi}(\widetilde{\mathcal{M}},\tau)$}

In this section, we prove that the noncommutative Orlicz spaces $L_{\varphi}(\widetilde{\mathcal{M}},\tau)$ with the Luxemburg norm have the Fatou property. Consequently, the space is complete. In addition, we give a new description of the subspace $E_{\varphi}(\widetilde{\mathcal{M}},\tau)$ given in \cite{Ghadir}, and prove that this is a closed linear subspace in norm topology and a dense subspace in measure topology of the $L_{\varphi}(\widetilde{\mathcal{M}},\tau)$.

Firstly we  give the definition of rearrangement invariant as follows.

\begin{definition}A linear subspace $E$ of $\widetilde{\mathcal{M}}$ is called
rearrangement invariant if and only if $x\in E, \, y\in\widetilde{\mathcal{M}}$ and for all $t>0$, $\mu_{t}(y)\leq\mu_{t}(x)$ imply that $ y\in E$ and $\|y\|_{E}\leq\|x\|_{E}.$
\end{definition}

It is well known that $L_{\varphi}(\widetilde{\mathcal{M}},\tau)$ are normed rearrangement invariant operator spaces \cite{Bennet}. From Corollary 2.4 in \cite{Dodds} we know that a normed rearrangement invariant operator space with the Fatou property is a Banach space. Hence, in order to prove that $L_{\varphi}(\widetilde{\mathcal{M}},\tau)$ are  Banach spaces, it suffices to show that $L_{\varphi}(\widetilde{\mathcal{M}},\tau)$ satisfy the Fatou property.

\begin{theorem}
(Fatou property) Suppose that $x\in\widetilde{\mathcal{M}}, x_{n}\in L_{\varphi}(\widetilde{\mathcal{M}},\tau)$. If $\sup_{n}\|x_{n}\|<\infty$ and $0\leq x_{n}\uparrow_{n} x$, then $x\in L_{\varphi}(\widetilde{\mathcal{M}},\tau)$ and $\|x\|=\sup_{n}\|x_{n}\|.$
\end{theorem}
\begin{proof}
Since $x_{n}\in L_{\varphi}(\widetilde{\mathcal{M}},\tau)$, one has that $\mu(x_{n})\in L_{\varphi}(0,\infty)$. From Proposition 1.7 in \cite{Dodds}, if $x_{n}, \, x\in\widetilde{\mathcal{M}}$ and $0\leq x_{n}\uparrow_{n} x$ then $\mu_{t}(x_{n})\uparrow_{n} \mu_{t}(x)$ holds for all $t\geq0$.
Since $\sup_{n}\|x_{n}\|<\infty=\sup_{n}\|\mu(x_{n})\|<\infty$.

Then $\mu(x)\in L_{\varphi}(0,\infty)$ and $\|\mu(x)\|=\sup_{n}\|\mu(x_{n})\|$ by the classical counterpart of Theorem 2.1.

Hence, $x\in L_{\varphi}(\widetilde{\mathcal{M}},\tau)$ and $\|x\|=\sup_{n}\|x_{n}\|.$
\end{proof}
In \cite{Kunze}, Kunze considers  the properties of the space
$$E_{\varphi}(\widetilde{\mathcal{M}},\tau)=\overline{\mathcal{M}\bigcap L_{\varphi}(\widetilde{\mathcal{M}},\tau)}^{\|\cdot\|}.$$
Now we give another characterization of this space. Indeed, set
$$A_{\varphi}(\widetilde{\mathcal{M}},\tau)=\left\{x\in\widetilde{\mathcal{M}}: \tau\left(\varphi\left(\lambda |x|\right)\right)<\infty\,\, for\,\, all\,\,\lambda>0\right\}.$$
It is easy to verify that $A_{\varphi}(\widetilde{\mathcal{M}},\tau)$ is a linear subspace of $L_{\varphi}(\widetilde{\mathcal{M}},\tau)$.  The following theorem shows that $A_{\varphi}(\widetilde{\mathcal{M}},\tau)$ is a closed linear subspace in norm topology and a dense subspace in measure topology of $L_{\varphi}(\widetilde{\mathcal{M}},\tau)$.
\begin{theorem}
The following statements are true:
\begin{enumerate}
  \item $A_{\varphi}(\widetilde{\mathcal{M}},\tau)$ is a closed linear subspace of $L_{\varphi}(\widetilde{\mathcal{M}},\tau)$ under the norm topology;
  \item $A_{\varphi}(\widetilde{\mathcal{M}},\tau)$ is a dense subspace of $L_{\varphi}(\widetilde{\mathcal{M}},\tau)$ under the measure topology.
\end{enumerate}
\end{theorem}
\begin{proof}
(1) Let $x_n\in A_{\varphi}(\widetilde{\mathcal{M}},\tau)$ and $x \in L_{\varphi}(\widetilde{\mathcal{M}},\tau)$ be given with $x_{n}\rightarrow x$ in follows from Lemma of \cite{Labuschagne} that any $z\in \mathcal{M}$ belongs to $A_{\varphi}(\widetilde{\mathcal{M}})$ if and only if $\mu(z)\in A_{\varphi}(0,\infty)$.

By corollary 4.3 of \cite{Dodds1} we now have that
$$\|\mu(x_{n})-\mu(x)\|\leq \|\mu(x_{n}-x)\|=\|x_{n}-x\|\rightarrow 0.$$
By the classical counterpart of Theorem 2.1 we now have that $\mu(x)\in A_{\varphi}(0,\infty).$

Hence $x\in A_{\varphi}(\widetilde{\mathcal{M}},\tau)$ as required.

(2) Now for any $x\in L_{\varphi}(\widetilde{\mathcal{M}},\tau)$, set
  $$x=u|x|=u\int_{0}^{\infty}\lambda \mathrm{d}e_{\lambda}(|x|)$$
 be the polar decomposition of $x$. For each $n\in\mathbb{N}$, set
$x_{n}=u\int_{0}^{n}\lambda \mathrm{d} e_{\lambda}(|x|),$ then it is obvious that $x_{n}\in A_{\varphi}(\widetilde{\mathcal{M}},\tau)$ and $$x-x_{n}=u|x|e_{(n,\infty)}(|x|)=u\int_{n}^{\infty}\lambda \mathrm{d}e_{\lambda}(|x|).$$

For any $\varepsilon>0,$
\begin{equation}
e_{(\varepsilon,\infty)}(|x-x_{n}|)=
\begin{cases} e_{(n,\infty)}(|x|), &\varepsilon<n, \nonumber\\
e_{(\varepsilon,\infty)}(|x|), &\varepsilon\geq n.
\end{cases}
\end{equation}
Since $x$ is a $\tau$-measurable operator, $\lim_{n}\tau(e_{(n,\infty)}(|x|))=0$, which means that $\lim_{n}\tau(e_{(\varepsilon,\infty)}(|x-x_{n}|))=0$ for any $\varepsilon>0$.
Hence, $A_{\varphi}(\widetilde{\mathcal{M}},\tau)$ is a dense subspace of $L_{\varphi}(\widetilde{\mathcal{M}},\tau)$ under the measure topology.
\end{proof}

In order to study the further properties of $A_{\varphi}(\widetilde{\mathcal{M}},\tau)$, we need the following lemma.

\begin{lemma} By $E_{\varphi}(\widetilde{\mathcal{M}},\tau)=E_{\varphi}$ we denote the set $\overline{\mathcal{M}\bigcap L_{\varphi}(\widetilde{\mathcal{M}},\tau)}^{\|\cdot\|}$.
If $x\in L_{\varphi}(\widetilde{\mathcal{M}},\tau)$ and $\tau(\varphi(|x|))<\infty$, then the distance $d(x,E_{\varphi})$ from $x$ to $E_{\varphi}$ is no more than 1, where $d(x,E_{\varphi})=\inf\left\{\|x-y\|:\,\,y\in E_{\varphi}\right\}.$
\end{lemma}
\begin{proof}
Let $x=u|x|$ be the polar decomposition of $x$, where $|x|=\int_{0}^{\infty}\lambda \mathrm{d}e_{\lambda}(|x|)$. For each $n\in\mathbb{N}$, set
$x_{n}=u\int_{0}^{n}\lambda \mathrm{d}e_{\lambda}(|x|).$
Since $\tau(\varphi(|x|))<\infty$, for any $\varepsilon>0$ one can choose an $n_{0}\in\mathbb{N}$ such that $$\tau(\varphi(|x-x_{n_{0}}|)=\int_{n_{0}}^{\infty}\varphi(\lambda) \mathrm{d}\tau(e_{\lambda})<\varepsilon.$$
Since $x_{n}\in E_{\varphi},$\, the Young Inequality implies
$$d(x,E_{\varphi})\leq \|x-x_{n_{0}}\|\leq\|x-x_{n_{0}}\|^{o}\leq1+\tau(\varphi(|x-x_{n_{0}}|)<1+\varepsilon.$$
Therefore, $d(x, E_{\varphi})\leq1$ since $\varepsilon$ is arbitrary.
\end{proof}

The following theorem shows that $A_{\varphi}(\widetilde{\mathcal{M}},\tau)$ is the closure (in the norm topology) of the set of all bounded $\tau$-measurable operators.

\begin{theorem}
$A_{\varphi}(\widetilde{\mathcal{M}},\tau)=E_{\varphi}(\widetilde{\mathcal{M}},\tau)=\overline{\mathcal{M}\bigcap L_{\varphi}(\widetilde{\mathcal{M}},\tau)}^{\|\cdot\|}.$
\end{theorem}

\begin{proof}
For any $x\in A_{\varphi}(\widetilde{\mathcal{M}},\tau)$ and $k\geq1$, we have $kx\in A_{\varphi}(\widetilde{\mathcal{M}},\tau).$ Therefore $d(kx,E_{\varphi})\leq1$ or $d(x,E_{\varphi})\leq\frac{1}{k}$. Since $k$ is arbitrary, then we have $x\in E_{\varphi}$, $i.e.$, $A_{\varphi}(\widetilde{\mathcal{M}},\tau)\subseteq E_{\varphi}$.

On the other hand, observing that $\mathcal{M}$ is contained in $A_{\varphi}(\widetilde{\mathcal{M}},\tau)$ and that $A_{\varphi}(\widetilde{\mathcal{M}},\tau)$ is a closed subspace of $L_{\varphi}(\widetilde{\mathcal{M}},\tau)$ by (1) of Theorem 2.2, then $E_{\varphi}$ is contained in $A_{\varphi}(\widetilde{\mathcal{M}},\tau)$, which implies that $A_{\varphi}(\widetilde{\mathcal{M}},\tau)=E_{\varphi}.$

Moreover, by the definition of $A_{\varphi}(\widetilde{\mathcal{M}},\tau)$, we get $$A_{\varphi}(\widetilde{\mathcal{M}},\tau)=\overline{\mathcal{M}\bigcap L_{\varphi}(\widetilde{\mathcal{M}},\tau)}^{\|\cdot\|}=E_{\varphi}(\widetilde{\mathcal{M}},\tau).$$
\end{proof}
In the following, similar to the classical case, we still use $E_{\varphi}(\widetilde{\mathcal{M}},\tau)$ to denote the set $\left\{x\in\widetilde{\mathcal{M}}: \tau\left(\varphi\left(\lambda |x|\right)\right)<\infty\,\, for\,\, all\,\,\lambda>0\right\}.$
\begin{theorem}
Let $x\in \widetilde{\mathcal{M}}$ be given, the following statements are equivalent:
\begin{enumerate}
  \item $x\in E_{\varphi}(\widetilde{\mathcal{M}},\tau)$.
  \item $\lim_{n\rightarrow\infty}\|x-x_{n}\|=0$,\, where $x_{n}=u\int_{0}^{n}\lambda  \mathrm{d} e_{\lambda}(|x|).$
\end{enumerate}
\end{theorem}
\begin{proof}
$(1)\Rightarrow(2)$. Given $\varepsilon>0,$ we have $$\tau\left(\varphi\left(\frac{|x|}{\varepsilon}\right)\right)=\int^{\infty}_{0}\varphi\left(\frac{\lambda}{\varepsilon}\right) \mathrm{d}\tau(e_{\lambda}(|x|))<\infty,$$
this is because $x\in E_{\varphi}(\widetilde{\mathcal{M}},\tau)$.

Since $x_{n}=u\int_{0}^{n}\lambda \mathrm{d}e_{\lambda}(|x|),$ when $n$ is large enough it follows that $$\tau\left(\varphi\left(\frac{|x-x_{n}|}{\varepsilon}\right)\right)=\int^{\infty}_{n}\varphi\left(\frac{\lambda}{\varepsilon}\right) \mathrm{d}\tau(e_{\lambda}(|x|))\leq1.$$

Hence, by the Young Inequality,
$$\left\|\frac{x-x_{n}}{\varepsilon}\right\|\leq\left\|\frac{x-x_{n}}{\varepsilon}\right\|^{o}\leq1
+\tau\left(\varphi\left(\frac{|x-x_{n}|}{\varepsilon}\right)\right)\leq2$$
for such $n\in\mathbb{N}$. This yields $\|x-x_{n}\|\rightarrow0$ as $n\rightarrow\infty$, since $\varepsilon$ is arbitrary.

$(2)\Rightarrow(1)$. Since $x_{n}=u\int_{0}^{n}\lambda \mathrm{d}e_{\lambda}(|x|)$ , then $x_{n}\in\mathcal{M}$.
If $\lim_{n\rightarrow \infty}\|x-x_{n}\|=0,$ then $x\in E_{\varphi}=E_{\varphi}(\widetilde{\mathcal{M}},\tau)$ by definition.
\end{proof}

\section{The properties of $L_{\varphi}(\widetilde{\mathcal{M}},\tau)$ for $\varphi\in \Delta_2$}

In this section, we will prove that if the Orlicz function $\varphi$ satisfies the $\Delta _2$ condition (for short, denote it by $\varphi\in \Delta_2$), namely, there exists a constant $k>0$ such that for all $u>0$,
$$\varphi(2u)\leq k\varphi(u),$$
then $E_{\varphi}(\widetilde{\mathcal{M}},\tau)$ is uniformly monotone, and
$$E_{\varphi}(\widetilde{\mathcal{M}},\tau)=L_{\varphi}(\widetilde{\mathcal{M}},\tau).$$
Using Lemma 2.1 of \cite{Labuschagne}, and the fact that $\|\mu(x)\|=\|x\|$ we can get following Lemmas from the classical counterparts of these Lemmas applied to $\mu(x)$.
\begin{lemma}
Suppose $\varphi\in\Delta_{2}$ and $x\in L_{\varphi}(\widetilde{\mathcal{M}},\tau).$ For any $\varepsilon>0,$ there exists a $\delta(\varepsilon)>0$ such that $\tau(\varphi(|x|))\geq\delta$ whenever $\|x\|\geq\varepsilon.$
\end{lemma}

\begin{lemma}
Suppose $\varphi\in\Delta_{2}$ and $x\in L_{\varphi}(\widetilde{\mathcal{M}},\tau).$ For any $\varepsilon\in(0,1),$ there exists a $\delta(\varepsilon)\in(0,1)$ such that $\|x\|\leq1-\delta$ whenever $\tau(\varphi(|x|))\leq1-\varepsilon.$
\end{lemma}

\begin{lemma}
Suppose $\varphi\in\Delta_{2}$ and $x\in L_{\varphi}(\widetilde{\mathcal{M}},\tau).$ For any $\varepsilon\in(0,1),$ there exists a $\delta(\varepsilon)\in(0,1)$ such that $\|x\|\geq1+\delta$ whenever $\tau(\varphi(|x|))\geq1+\varepsilon.$
\end{lemma}

\begin{theorem}
Assume $\varphi\in\Delta_{2}$.  Given any $L>0$ and $\varepsilon>0,$ there exists $\delta(L,\varepsilon)>0$ such that
$\tau(\varphi(|x|))\leq L$ and $ \tau(\varphi(|y|))\leq\delta$, implies that $$|\tau(\varphi(|x+y|))-\tau(\varphi(|x|))|<\varepsilon.$$
\end{theorem}

\begin{proof}
Firstly, for any $x\in L_{\varphi}(\widetilde{\mathcal{M}},\tau)$ and  $\|x\|\leq1$.

 Since $\|\mu(x)\|=\|x\|$, by the classical counterpart,
 \begin{eqnarray*}
 \tau(\varphi(x))&=&\int_{0}^{\infty}\varphi(\mu_{t}(x))dt\\
 &\leq& \|\mu_{t}(x)\|=\|x\|.
 \end{eqnarray*}
Now set
$$h=\sup\{\tau(\varphi(|2x|+|2y|)): \tau(\varphi(|x|))\leq L, \tau(\varphi(|y|))\leq 1\}.$$
Then $L<h<\infty$ since $\varphi\in\Delta_{2}.$ Without loss of generality, we can assume $L>1$ and $\varepsilon<1$.

Set $\beta=\frac{\varepsilon}{h}$, by Lemma 3.2, there exists a $\delta>0$ such that $\tau(\varphi(|y|))\leq \delta$ implies $\|y\|\leq\min\{\frac{\beta}{2},\frac{\varepsilon}{2}\}$ $i.e.$, $\big\|\frac{2}{\beta}y\big\|\leq1$. Hence, if
$\tau(\varphi(|x|))\leq L$ and $\tau(\varphi(|y|))\leq\delta$, then by (iii) of Theorem 4.4 in \cite{Fack} and convexity of $\varphi$,
\begin{eqnarray*}
\tau(\varphi(|x+y|))&=&\int^{\infty}_{0}\varphi\left(\mu_{t}(|x+y|)\right) \mathrm{d}t\\
&\leq&\int^{\infty}_{0}\varphi\left(\mu_{t}(u|x|u^{\ast}+v|y|v^{\ast})\right) \mathrm{d}t\\
&\leq&\int^{\infty}_{0}\varphi\left(\mu_{t}(u|x|u^{\ast})+\mu_{t}(v|y|v^{\ast})\right) \mathrm{d}t\\
&\leq&\int^{\infty}_{0}\varphi\left(\mu_{t}(|x|)+\mu_{t}(|y|)\right) \mathrm{d}t\\
&=&\int^{\infty}_{0}\varphi\left((1-\beta)\mu_{t}(|x|)+\beta\left(\mu_{t}(|x|)+\frac{\mu_{t}(|y|)}{\beta}\right)\right) \mathrm{d}t\\
&\leq&(1-\beta)\int^{\infty}_{0}\varphi(\mu_{t}(|x|)) \mathrm{d}t+\beta\int^{\infty}_{0}\varphi\left(\mu_{t}(|x|)+\frac{\mu_{t}(|y|)}{\beta}\right) \mathrm{d}t\\
&\leq&(1-\beta)\int^{\infty}_{0}\varphi\left(\mu_{t}\left(|x|\right)\right) \mathrm{d}t+\frac{\beta}{2}\Big[\int^{\infty}_{0}\varphi(\mu_{t}(2|x|)) \mathrm{d}t\\
&&+\int^{\infty}_{0}\varphi \left(\mu_{t} \left(\frac{2|y|}{\beta} \right) \right) \mathrm{d}t \Big]\\
&=&(1-\beta)\tau(\varphi(|x|))+\frac{\beta}{2}\left[\tau(\varphi(2|x|))+\tau\left(\varphi\left(\frac{2|y|}{\beta}\right)
\right)\right]\\
&\leq&\tau(\varphi(|x|))+\frac{\beta h}{2}+\frac{\beta}{2}\left\|\frac{2}{\beta}y\right\|\\
&\leq&\tau(\varphi(|x|))+\varepsilon.
\end{eqnarray*}
Respectively, replacing $x,y$ by $x+y,-y$ in the above inequalities, we also have
$$\tau(\varphi(|x|))=\tau(\varphi(|(x+y)+(-y)|))\leq\tau(\varphi(|x+y|))+\varepsilon.$$
This completes the proof.
\end{proof}

We say a noncommutative Orlicz space $L_{\varphi}(\widetilde{\mathcal{M}},\tau)$ is uniformly monotone, if for each $\varepsilon>0$ there exists $\delta(\varepsilon)>0$ such that for positive $\tau$-measurable operators $x,y$ with $\|x\|=1$ and $\|y\|\geq\varepsilon$, we have $\|x+y\|\geq1+\delta(\varepsilon).$
\begin{theorem}
If $\varphi\in\Delta_{2},$ then $ L_{\varphi}(\widetilde{\mathcal{M}},\tau)$ is uniformly monotone.
\end{theorem}
\begin{proof}
Let $\varepsilon>0$ and $x,y\in L^{+}_{\varphi}(\widetilde{\mathcal{M}},\tau)$ such that $\|x\|=1$ and $\|y\|\geq\varepsilon$. From (ii) of Proposition 3.6 in \cite{Ghadir} we have $\tau(\varphi(x))=1$ since $\varphi\in\Delta_{2}$ and from Lemma 3.1 we have that $\tau(\varphi(y))\geq\eta$ where $\eta=\eta(\varepsilon)>0$ is as in hypothesis of Lemma 3.1.

Then by (ii) of Proposition 4.6 in \cite{Fack} we get
$$\tau(\varphi(x+y))\geq\tau(\varphi(x))+\tau(\varphi(y))\geq1+\eta.$$
Hence, using Lemma 3.3, there exists a $\delta>0$ such that $\|x+y\|\geq1+\delta.$
\end{proof}
The following theorem shows that under the condition $\varphi\in\Delta_{2}$,  convergence in norm and in measure coincide on the unit sphere of $(L_{\varphi}(\widetilde{\mathcal{M}},\tau),\|\cdot\|).$

\begin{theorem}
Assume $x_{n},x\in L_{\varphi}(\widetilde{\mathcal{M}},\tau)$. If
$\lim_{n\rightarrow \infty}\tau(\varphi(|x_{n}|))= \tau(\varphi(|x|))$ and $x_{n}\xrightarrow{\tau_{m}}x$, then $ \lim_{n\rightarrow \infty}\tau\big(\varphi\big(\big|\frac{x_{n}-x}{2}\big|\big)\big)=0.$
Moreover, if in addition, $\varphi\in\Delta_{2}$, then $\|x_{n}-x\|\rightarrow0.$
\end{theorem}

\begin{proof}
By the convexity of $\varphi$  and (ii),(v) of Lemma 2.5 in \cite{Fack}, we have
\begin{eqnarray*}
\varphi\left(\mu_{t}\left(\frac{|x-x_{n}|}{2}\right)\right)&=&\varphi\left(\frac{1}{2}\mu_{t}(|x-x_{n}|)\right)\\
&\leq&\varphi\left(\frac{1}{2}\mu_{\frac{t}{2}}\left(|x|\right)+\frac{1}{2}\mu_{\frac{t}{2}}\left(|x_{n}|\right)\right]\\
&\leq&\frac{1}{2}\left[\varphi\left(\mu_{\frac{t}{2}}\left(|x|\right)\right)+\varphi\left(\mu_{\frac{t}{2}}\left(|x_{n}|\right)\right)\right].
\end{eqnarray*}

If $x_{n}\xrightarrow{\tau_{m}}x$, it follows from Lemma 3.1 of \cite{Fack} that $\lim_{n\rightarrow\infty}\mu_{t}(x_{n}-x)=0$ for each $t>0$.
Suppose that $\tau(\varphi(|x_{n}|))\rightarrow \tau(\varphi(|x|))$,\\
Fatou's Lemma implies
\begin{eqnarray*}
\int^{\infty}_{0}\varphi\left(\mu_{\frac{t}{2}}\left(|x|\right)\right) \mathrm{d}t
&=&\int^{\infty}_{0}\lim_{n\rightarrow\infty}\Big[\frac{\varphi\left(\mu_{\frac{t}{2}}\left(|x|\right)\right)+\varphi\left(\mu_{\frac{t}{2}}\left(|x_{n}|\right)\right)}{2}\\
&&-\varphi\left(\mu_{t}\left(\frac{|x-x_{n}|}{2}\right)\right)\Big] \mathrm{d}t\\
&\leq&\lim_{n\rightarrow\infty}\inf\int^{\infty}_{0}\Big[\frac{\varphi\left(\mu_{\frac{t}{2}}\left(|x|\right)\right)+\varphi\left(\mu_{\frac{t}{2}}\left(|x_{n}|\right)\right)}{2}\\
&&-\varphi\left(\mu_{t}\left(\frac{|x-x_{n}|}{2}\right)\right)\Big] \mathrm{d}t\\
&=&\int^{\infty}_{0}\varphi\left(\mu_{\frac{t}{2}}\left(|x|\right)\right) \mathrm{d}t-\lim_{n\rightarrow\infty}\sup\tau\left(\varphi\left(\frac{|x-x_{n}|}{2}\right)\right).
\end{eqnarray*}
Then we obtain $$-\lim_{n\rightarrow\infty}\sup\tau\left(\varphi\left(\frac{|x-x_{n}|}{2}\right)\right)\geq0,$$
which implies $\tau\left(\varphi\left(\frac{|x_{n}-x|}{2}\right)\right)\rightarrow0$.
Hence $\|x_{n}-x\|\rightarrow0$ since $\varphi\in\Delta_{2}$.
\end{proof}

From Theorem 2.2 we know $E_{\varphi}(\widetilde{\mathcal{M}},\tau)$ is a closed linear subspace in norm topology and a dense subspace of $L_{\varphi}(\widetilde{\mathcal{M}},\tau)$ in measure topology. The next theorem shows that $L_{\varphi}(\widetilde{\mathcal{M}},\tau)=E_{\varphi}(\widetilde{\mathcal{M}},\tau)=\overline{\mathcal{M}}$,
when $\varphi\in\Delta_{2}$.

\begin{theorem}
If $\varphi\in\Delta_{2}$, then $$E_{\varphi}(\widetilde{\mathcal{M}},\tau)=L_{\varphi}(\widetilde{\mathcal{M}},\tau).$$
\end{theorem}

\begin{proof}
By Theorem 3.3, under the condition $\varphi\in\Delta_{2}$, the convergence in norm topology and in measure topology coincide. Hence, we can get the conclusion.

\end{proof}
\begin{corollary}
For any $\tau$-measurable operator $x\in{\widetilde{\mathcal{M}}}$, suppose that $\varphi(x)=|x|^{p},\,\,1\leq p<\infty$, then $\varphi\in\Delta_{2}$. Therefore,
\begin{eqnarray*}
L^{p}(\widetilde{\mathcal{M}},\tau)&=&E^{p}(\widetilde{\mathcal{M}},\tau)\\
&=&\left\{x\in\widetilde{\mathcal{M}}: \tau\left((\lambda |x|)^{p}\right)<\infty\,\, for\,\, all\,\,\lambda>0\right\}\\
&=&\left\{x\in\widetilde{\mathcal{M}}: \tau\left(|x|^{p}\right)<\infty\right\}.
\end{eqnarray*}
\end{corollary}
\begin{remark}
Notice that the condition $\varphi\in\Delta_{2}$ in theorem 3.4 is necessary, that is to say , if $\varphi\notin\Delta_{2}$, $E_{\varphi}(\widetilde{\mathcal{M}},\tau)\subsetneqq L_{\varphi}(\widetilde{\mathcal{M}},\tau)$.
Indeed, suppose $\varphi\notin\Delta_{2}$. By $(2)$ of Theorem $1.13$ in \cite{Chen}, there exist $0<\alpha_{k}\uparrow\infty$ such that
$$\varphi\left(\left(1+\frac{1}{k}\alpha_{k}\right)\right)>2^{k}\varphi(\alpha_{k})\,\,\, (k\in\mathbb{N}).$$
Select mutually orthogonal projections $\{e_{k}\}$ in a non-atomic von Neumann algebra $\mathcal{M}$ such that
$\varphi(\alpha_{k})\tau(e_{k})=\frac{\varepsilon}{2^{k}},$
where $\varepsilon>0$ and $k\in\mathbb{N}$,
and define $$x_{n}=\sum_{k=n+1}^{\infty}\alpha_{k}e_{k}$$
with $\alpha_{k}\in\mathbb{R_{+}}$.
Then,
$$\tau(\varphi(x_{n}))=\sum_{k=n+1}^{\infty}\varphi\left(\alpha_{k}\right)\tau\left(e_{k}\right)=\frac{\varepsilon}{2^{n}}<\infty,$$  which means $x_{n}\in L_{\varphi}(\widetilde{\mathcal{M}},\tau).$

But for any $l>1,$ let $n_{0}\in\mathbb{N}$ satisfy $l\geq1+\frac{1}{n_{0}}.$ Then for any $n\geq n_{0},$
\begin{eqnarray*}
\tau(\varphi(l x_{n}))&>&\sum_{k=n+1}^{\infty}\varphi\left(\left(1+\frac{1}{k}\alpha_{k}\right)\right)\tau(e_{k})\\
&>&\sum_{k=n+1}^{\infty}2^{k}\varphi\left(\alpha_{k}\right)\tau(e_{k})\\
&=&\sum_{k=n+1}^{\infty}\varepsilon=\infty,
\end{eqnarray*}
which shows that $x_{n}\notin E_{\varphi}(\widetilde{\mathcal{M}},\tau)\,\,\,(n\in \mathbb{N}).$ Therefore if $\varphi\notin\Delta_{2}$,
$E_{\varphi}(\widetilde{\mathcal{M}},\tau)\subsetneqq L_{\varphi}(\widetilde{\mathcal{M}},\tau)$.
\end{remark}
\section*{Acknowledgement} We want to express our gratitude to the referee for all his/her careful revision and suggestions which has improved the final version of this work.

\section*{References}

\bibliography{mybibfile}

\end{document}